\documentclass[a4paper]{amsart}
\usepackage{fixltx2e}
\usepackage[T1]{fontenc}
\usepackage{amssymb,amsthm,amsmath}
\usepackage[british]{babel}
\usepackage[all]{xy}
\usepackage{mathabx}
\usepackage{calrsfs}
\usepackage{graphicx}

\theoremstyle{plain}
\newtheorem*{thm}{Hagemann's Theorem}
\newtheorem{theorem}{Theorem}
\newtheorem{proposition}{Proposition}

\theoremstyle{definition}
\newtheorem{example}{Example}

\newcommand{\defn}{\textbf}
\newcommand{\comp}{\raisebox{0.2mm}{\ensuremath{\scriptstyle{\circ}}}}

\newcommand{\CC}{\ensuremath{\mathcal{C}}}
\newcommand{\VV}{\ensuremath{\mathcal{V}}}
\newcommand{\WW}{\ensuremath{\mathcal{W}}}

\newcommand{\noproof}{\hfill \qed}

\newcommand{\nPerm}[1]{\ensuremath{\text{$#1$-\textsf{Perm}}}}
\newcommand{\Perm}{\ensuremath{\mathsf{Perm}}}

\newdir{>}{{}*:(1,-.2)@^{>}*:(1,+.2)@_{>}}
\newdir{<}{{}*:(1,+.2)@^{<}*:(1,-.2)@_{<}}

\def\pullback{
 \ar@{-}[]+R+<6pt,-1pt>;[]+RD+<6pt,-6pt>%
 \ar@{-}[]+D+<1pt,-6pt>;[]+RD+<6pt,-6pt>}

\begin{document}

\title{An observation on $n$-permutability}

\author{Nelson Martins-Ferreira}
\author[Diana Rodelo]{Diana Rodelo}
\author[Tim Van~der Linden]{Tim Van~der Linden}

\email{martins.ferreira@ipleiria.pt}
\email{drodelo@ualg.pt}
\email{tim.vanderlinden@uclouvain.be}

\address{Departamento de Matem\'atica, Escola Superior de Tecnologia e Gest\~ao, Centro para o Desenvolvimento R\'apido e Sustentado do Produto, Instituto Poli\-t\'ecnico de Leiria, Leiria, Portugal}
\address{CMUC, Universidade de Coimbra, 3001--454 Coimbra, Portugal}\address{Departamento de Matem\'atica, Faculdade de Ci\^{e}ncias e Tecnologia, Universidade do Algarve, Campus de
Gambelas, 8005--139 Faro, Portugal}
\address{Institut de recherche en math\'ematique et physique, Universit\'e catholique de Louvain, che\-min du cyclotron~2 bte~L7.01.02, B--1348 Louvain-la-Neuve, Belgium}

\thanks{The first author was supported by IPLeiria/ESTG-CDRSP and Funda\c c\~ao para a Ci\^encia e a Tecnologia (under grant number SFRH/BPD/4321/2008). The second author's research was supported by CMUC, funded by the European 
Regional Development Fund through the program COMPETE and by the Portuguese Government through the FCT-Funda\c c\~ ao para a Ci\^encia e a Tecnologia under the project 
PEst-C/MAT/UI0324/2011. The third author works as \emph{charg\'e de recherches} for Fonds de la Recherche Scientifique--FNRS and would like to thank CMUC for its kind hospitality during his stay in Coimbra. All three were supported by the FCT Grant PTDC/MAT/120222/2010 through the European program COMPETE/FEDER}



\begin{abstract}
We prove that in a regular category all reflexive and transitive relations are symmetric if and only if every internal category is an internal groupoid. In particular, these conditions hold when the category is $n$-permutable for some~$n$.
\end{abstract}

\date{\today}

\maketitle

Let $\CC$ be a regular category. It is well known that any internal preorder, being a reflexive and transitive relation $(R,r_{1},r_{2})$ on an object $X$ of $\CC$, may be considered as an internal category in $\CC$. In fact, a preorder is the same thing as a \emph{thin} category, an internal category of which the domain and codomain morphisms $r_{1}$, $r_{2}\colon {R\to X}$ are jointly monic. This internal category will be a groupoid precisely when the given reflexive and transitive relation $R$ is symmetric, so that \emph{if in~$\CC$ every internal category is an internal groupoid, then all of its internal reflexive and transitive relations are equivalence relations}.

The converse implication is interesting due to its close relation with the following question: what conditions does a regular category need to satisfy for all internal categories in it to be internal groupoids? One of the main results of~\cite{Carboni-Pedicchio-Pirovano} gives a sufficient condition: the Mal'tsev property, that is, $2$-permutability $RS=SR$ of internal equivalence relations $R$, $S$ on the same object. But when~$\CC$ is a variety, already the strictly weaker $n$-permutability condition ($RSRS\ldots=SRSR\ldots$ with~$n$ factors $R$ or $S$ on each side) is sufficient~\cite{Rodelo:Internal-categories}. Furthermore---here we follow a remark in~\cite{MFVdL2}---a variety is $n$-permutable if and only if~\cite{CR} all of its internal reflexive and transitive relations are equivalence relations (= congruences). Altogether:

\begin{proposition}
\label{1}
If $\CC$ is a variety of universal algebras, then the following conditions are equivalent:
\begin{enumerate}
\item all preorders in $\CC$ are congruences;
\item all internal categories in $\CC$ are internal groupoids;
\item $\CC$ is $n$-permutable for some $n\geq 2$.\noproof
\end{enumerate}
\end{proposition}

This result is no longer true for regular categories. The number $n$ in the third condition is obtained through a construction on a free algebra, and it cannot be replaced by a purely categorical argument, as shows the following counterexample.

\begin{example}\label{Example on permutability algebras}
Consider the product category
\[
\Perm=\prod_{n\geq 2}\bigl(\nPerm{n}\bigr)=\nPerm{2}\times \nPerm{3}\times \cdots\times \nPerm{n}\times \cdots
\]
where, for $n\geq 2$, we let $\nPerm{n}$ be the ($n$-permutable) variety of \defn{$n$-permutability algebras} with operations $\theta_{1}$, \dots, $\theta_{n-1}$ for which the identities
\[
\begin{cases}
 \theta_1(s,t,t)=s,\\
 \theta_{i}(s,s,t) = \theta_{i+1}(s,t,t), & \text{for $i\in\{1, \dots, n-2\}$,}\\
 \theta_{n-1}(s,s,t)=t
\end{cases}
\]
hold. 

It is easy to see that $\Perm$ is a regular category. It is also clear that in $\Perm$, all preorders are equivalence relations: each of its components lies in some variety $\nPerm{n}$, where it will be a congruence. On the other hand, there is no $n\geq 2$ for which the category $\Perm$ is $n$-permutable, since otherwise $\nPerm{(n+1)}$ would be an $n$-permutable variety. Indeed, for any $n$ there are examples of $(n+1)$-permutable varieties which are not $n$-permutable~\cite{Graetzer, Mitschke, Schmidt, Hagemann-Mitschke}, and by forgetting structure these counterexamples can be made to work here too.
\end{example}

On the other hand, the equivalence between the upper two conditions in the proposition makes sense in general and, given any $n$-permutable category, we may ask whether they hold or not. As it turns out, the situation is as good as it could possibly be. The following characterisation of $n$-permutability due to Hagemann~\cite{Hagemann, Hagemann-Mitschke} was recently extended from varieties to regular categories~\cite{JRVdL1}.

\begin{thm}\label{TheB}
For a regular category $\CC$, and a natural number ${n\geqslant 2}$, the following conditions are equivalent:
\begin{enumerate}
\item $\CC$ has $n$-permutable congruences;
\item $R^\circ\leqslant R^{n-1}$ for any internal reflexive relation $R$ in $\CC$;
\item $R^n\leqslant R^{n-1}$ for any internal reflexive relation $R$ in $\CC$.\noproof
\end{enumerate}
\end{thm}

This may now be used to obtain our main result.

\begin{theorem}\label{Main theorem}
If $\CC$ is a regular category, then the following conditions are equivalent:
\begin{enumerate}
\item all reflexive and transitive relations in $\CC$ are equivalence relations;
\item all internal categories in $\CC$ are internal groupoids.
\end{enumerate}
Furthermore, these conditions hold if $\CC$ is $n$-permutable for some $n\geq 2$.
\end{theorem}
\begin{proof}
We already recalled that the second condition is stronger than the first. For \mbox{(i) $\Rightarrow$ (ii)} it suffices to observe that the argument given by Carboni, Pedicchio and Pirovano in the Mal'tsev context~\cite[Theorem~2.2]{Carboni-Pedicchio-Pirovano} may be adapted to hold in regular categories. Their proof uses difunctionality of internal relations where we can use image factorisations.

Consider an internal category
\[
\xymatrix{M*M \ar[r]^-{m} & M \ar@<1ex>[r]^-{d} \ar@<-1ex>[r]_-{c} & O \ar[l]|-{i}}
\]
where $M*M$, the \emph{object of composable pairs} $\{\langle\beta,\gamma\rangle\mid c\beta=d\gamma\}$, denotes the pullback of $c$ and $d$, while the morphism $m$ is the composition. The image of the span
\[
\xymatrix{& M*M \ar[ld]_-{\pi_1} \ar[rd]^-{m}\\
M && M}
\]
is a relation on $M$ which we write $S$. Using generalised elements as in~\cite{Carboni-Kelly-Pedicchio}, it makes sense to say as on page~103 of~\cite{Carboni-Pedicchio-Pirovano} that a couple of arrows $\langle\beta,\alpha\rangle\colon {X\to M\times M}$ is in $S$ if and only if there exists an arrow $\gamma$ in $M$ for which $\gamma\comp\beta=\alpha$.
\[
\vcenter{\xymatrix@1@R=2.4495em@C=1.4142em{& {\cdot} \ar@{->}[ld]_-{\beta}\\
{\cdot} && {\cdot} \ar@{<-}[lu]_-{\alpha} \ar@{<--}[ll]^-{\gamma}}}
\]
More explicitly, there should exist a morphism $\gamma\colon Y\to M$ and a regular epimorphism $p\colon {Y\to X}$ such that ${m\langle \beta p,\gamma\rangle=\alpha p}$. In fact, as we shall see below, when   $m$ satisfies the \emph{left cancellation property}, we may choose $p=1_{X}$. 

The relation $S$ is not just reflexive as mentioned in~\cite{Carboni-Pedicchio-Pirovano}, but it is also transitive. Hence condition (i) tells us that $S$ is an equivalence relation on $M$. Suppose indeed that $\alpha$, $\beta$, $\delta\colon {X\to M}$ are such that $\langle\beta,\alpha\rangle$ and $\langle\delta,\beta\rangle$ are in~$S$. Then we have $\gamma$ and $p$ as above, and also a morphism $\epsilon\colon Y'\to M$ and a regular epimorphism $p'\colon Y'\to X$ with ${m\langle \delta p',\epsilon\rangle=\beta p'}$. Taking the pullback
\[
\xymatrix{Z \ar[r]^-{q} \ar[d]_-{q'} \pullback & Y \ar[d]^-{p} \\
Y' \ar[r]_-{p'} & X}
\]
of $p$ and $p'$ and writing $\mu=m\langle \epsilon q',\gamma q\rangle$ we may calculate
\begin{align*}
m\langle \delta p'q',\mu\rangle = m\langle \delta p'q',m\langle \epsilon q',\gamma q\rangle\rangle
&= m\langle m\langle\delta p',\epsilon\rangle q',\gamma q\rangle\\
&= m\langle \beta p' q',\gamma q\rangle
= m\langle \beta p,\gamma\rangle q
= \alpha p q = \alpha p' q'
\end{align*}
to see that $\langle\delta,\alpha\rangle$ is in $S$. It follows that $S$ is transitive.

Consider the composites $id$ and $ic\colon M\to M$. Given any $\alpha\colon {X\to M}$, the pair $\langle id\alpha,\alpha\rangle$ is in $S$. The symmetry of $S$ gives us $\langle \alpha, id\alpha\rangle$ in $S$, which yields a generalised element ${}^{\bullet}\alpha$ of $M$ such that ${}^{\bullet}\alpha\comp\alpha=id\alpha$  as above. Via an analogous argument we obtain a generalised element $\alpha^{\bullet}$ of $M$ satisfying $\alpha\comp\alpha^{\bullet}=ic\alpha$. More precisely, ${}^{\bullet}\alpha\colon{Y\to X}$ and $m\langle \alpha p,{}^{\bullet}\alpha\rangle=id\alpha p$ for some regular epimorphism $p\colon {Y\to X}$, while $\alpha^{\bullet}\colon {Y'\to X}$ and $m\langle \alpha^{\bullet},\alpha p'\rangle=ic\alpha p'$ for some regular epimorphism $p'\colon {Y'\to X}$. Taking again the above pullback of $p$ and~$p'$,
\begin{align*}
{}^{\bullet}\alpha q=m\langle ic\alpha p'q', {}^{\bullet}\alpha q\rangle &=m\langle m \langle \alpha^{\bullet}q',\alpha p'q'\rangle, {}^{\bullet}\alpha q\rangle\\
&=m\langle \alpha^{\bullet}q', m \langle \alpha p'q', {}^{\bullet}\alpha q\rangle\rangle
=m\langle \alpha^{\bullet}q', id\alpha p'q'\rangle
=\alpha^{\bullet}q'
\end{align*}
so $\overline\alpha\coloneq {}^{\bullet}\alpha q=\alpha^{\bullet}q'\colon {Z\to M}$, together with the regular epimorphism $pq\colon {Z\to X}$, is a two-sided inverse for $\alpha$.

We can use this to show that the composition satisfies the \defn{left cancellation property}: $\gamma\comp\beta=\gamma\comp\delta$ implies $\beta=\delta$. Given $\beta$, $\gamma$, $\delta\colon {X\to M}$ such that $c\beta=c\delta=d\gamma$, consider $\overline\gamma\colon Y\to M$ and the corresponding regular epimorphism $p\colon {Y\to X}$. The equality $m\langle\beta,\gamma\rangle=m\langle\delta,\gamma\rangle$ then implies
\begin{align*}
\beta p &= m\langle \beta p,id\gamma p\rangle
=m\langle \beta p,m\langle\gamma p, \overline\gamma\rangle\rangle
=m\langle m\langle\beta,\gamma\rangle p, \overline\gamma\rangle\\
&=m\langle m\langle\delta,\gamma\rangle p, \overline\gamma\rangle
=m\langle \delta p,m\langle\gamma p, \overline\gamma\rangle\rangle
=m\langle \delta p,id\gamma p\rangle
=\delta p,
\end{align*}
so $\beta=\delta$ as claimed.

The left cancellation property now allows us to lift the inverse $\overline\alpha\colon {Y\to M}$ of $\alpha\colon{X\to M}$ over the enlargement of domain $p\colon Y\to X$ which comes with it to a morphism $\alpha^{-1}\colon {X\to M}$.  To see this, consider the kernel relation $(R,\pi_{1},\pi_{2})$ of $p$
\[
\xymatrix{R \ar@<.5ex>[r]^-{\pi_{1}} \ar@<-.5ex>[r]_-{\pi_{2}} & Y \ar[r]^-{p} \ar[rd]_-{\overline\alpha} & X \ar@{.>}[d]^-{\alpha^{-1}} \\
&& M}
\]
and note that 
\[
m\langle \overline\alpha\pi_{1}, \alpha p\pi_{1}\rangle = ic\alpha p\pi_{1}
\qquad\text{and}\qquad
m\langle \overline\alpha\pi_{2}, \alpha p\pi_{2}\rangle = ic\alpha p\pi_{2}.
\]
Since $p\pi_{1}=p\pi_{2}$ by definition, left cancellation gives $\overline\alpha\pi_{1}=\overline\alpha\pi_{2}$, so that the morphism $\overline\alpha\colon {Y\to M}$ does indeed lift over $p$.

We now let $\alpha\colon {X\to M}$ be $1_{M}\colon {M\to M}$. The inverse $s=\alpha^{-1}=1_{M}^{-1}\colon {M\to M}$ is then a genuine inversion making the given internal category into a groupoid. This finishes the proof of~\mbox{(i) $\Rightarrow$ (ii)}.

For the final statement, suppose that $R$ is a reflexive and transitive relation. Then $R^{\circ}\leq R^{n-1}$ by Hagemann's Theorem while $R^{n-1}\leq R$ by transitivity of~$R$.
\end{proof}

In stark contrast with the above result, recall that a regular category is Mal'tsev if and only if every reflexive relation in it is an equivalence relation, while on the other hand, the so-called \defn{Lawvere condition} ``all internal reflexive graphs are internal groupoids'' means that the category is \emph{naturally Mal'tsev}~\cite{Carboni-Affine-Spaces, Johnstone:Maltsev}.  

We have just analysed the equivalence (A) in the picture
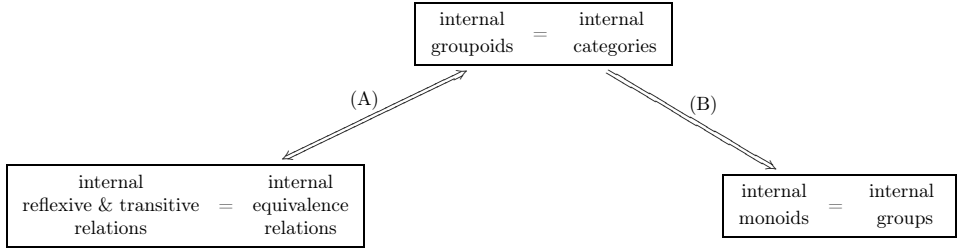
\[
\resizebox{\textwidth}{!}
{\xymatrix@=2em{ & \fbox{\begin{tabular}{ccc}internal & & internal \vspace{-5pt}\\ & = \vspace{-5pt}\\ groupoids & & categories\end{tabular}} \ar@{<=>}[ddl]_-{\txt{(A)}} \ar@{=>}[ddr]^-{\txt{(B)}} \\\\
\fbox{\begin{tabular}{ccc}internal & & internal \\ reflexive \& transitive & = & equivalence \\ relations & & relations\end{tabular}} && \fbox{\begin{tabular}{ccc}internal & & internal \vspace{-5pt}\\ & = \vspace{-5pt}\\ monoids & & groups\end{tabular}}}}
\]
and its relation with $n$-permutability. It is also clear that in any regular category which satisfies the equivalent conditions of Theorem~\ref{Main theorem}, all internal monoids are groups. One could now ask whether the implication (B) is also an equivalence and what is the role of $n$-permutability here.

By Theorem~1.4.5 in~\cite{Borceux-Bourn}, in a unital category, any internal monoid is commutative. Thus we can already conclude two things:
\begin{enumerate}
\item on the one hand, in a strongly unital category, any internal monoid is an abelian group~\cite[Theorem~1.9.5]{Borceux-Bourn};
\item on the other hand, if $\CC$ is regular and unital and the equivalent conditions of Theorem~\ref{Main theorem} hold, then in $\CC$ all internal monoids are abelian groups.
\end{enumerate}
So, any pointed Mal'tsev category, being strongly unital~\cite[Theorem~2.2.9]{Borceux-Bourn}, is such that every internal monoid in it is an internal abelian group. The same property holds for pointed Goursat (= $3$-permutable) categories~\cite[Corollary 3.4]{Bourn-Gran-Modularity}, even though these categories need not be (strongly) unital, as shows the following counterexample.

\begin{example}
We consider the variety $\VV$ of \defn{implication algebras}, which are $(I,\cdot)$ that satisfy
\[
\left\{\begin{aligned}
&(x y) x=x\\
&(x y) y=(y x) x \\
&x (y z)=y (x z)
\end{aligned}\right.
\]
where we write $x\cdot y=xy$. It is shown in~\cite{Mitschke, Hagemann-Mitschke} that $\VV$ is $3$-permutable. In particular, $\VV$ satisfies the equivalent conditions of Theorem~\ref{Main theorem}. In order to prove that $\VV$ is not unital, we construct a punctual span
\[
\xymatrix{X \ar@<-.5ex>[r]_-{s} & Z \ar@<-.5ex>[l]_-{f} \ar@<.5ex>[r]^-{g} & Y \ar@<.5ex>[l]^-{t}}
\]
as in~\cite[Theorem~1.2.12]{Borceux-Bourn} and such that the factorisation $\langle f,g\rangle\colon Z\to X\times Y$ is not a regular epimorphism. Put $X=\{1,2\}$, $Y=\{1,3\}$ and $Z=\{1,2,3\}$ with respective multiplication tables
\[
\begin{array}{c|cc}
 \cdot & 1 & 2 \\ \hline
 1 & 1 & 2 \\
 2 & 1 & 1
\end{array},
\qquad
\begin{array}{c|cc}
 \cdot & 1 & 3 \\ \hline
 1 & 1 & 3 \\
 3 & 1 & 1
\end{array}
\qquad\text{and}\qquad
\begin{array}{c|ccc}
 \cdot & 1 & 2 & 3 \\ \hline
 1 & 1 & 2 & 3 \\
 2 & 1 & 1 & 3 \\
 3 & 1 & 2 & 1
\end{array},
\]
take $s$ and $t$ to be the canonical inclusions and $f\colon{Z\to X}$ and $g\colon{Z \to Y}$ defined respectively by
\[
f(1)=f(3)=1,\quad f(2)=2
\]
and
\[
g(1)=g(2)=1,\quad g(3)=3.
\]
Then $\langle f,g\rangle$ is not a surjection, because $Z$ has three elements while $X\times Y$ has four.
\end{example}

Internal monoids in $n$-permutable varieties are always abelian groups. The proof uses arguments which are similar to the ones given in Proposition~5.3 of~\cite{Rodelo:Internal-categories}. The technique used in~\cite{JRVdL1} for transforming a varietal proof into a categorical one does not work in this specific situation, because the varietal proof uses nested operations.

Even in the context of varieties, implication (B) is generally not an equivalence. In fact, in~\cite{Barbour-Raftery} there are examples of subtractive varieties~\cite{Ursini3} which are not $n$-permutable for any~$n$. On the other hand, it is well known and easy to prove that in any subtractive variety, all internal monoids are abelian groups. 

Let indeed $\WW$ be a subtractive variety, so that it is pointed and admits a binary term $s$ satisfying $s(x,x)=0$ and $s(x,0)=0$. Let $(M,+)$ be a monoid in $\WW$. Then for any $x\in M$, the inverse of $x$ is $x^{\bullet}=s(0,x)$, so that $(M,+)$ is an internal group. It is also abelian:
\begin{multline*}
x+ y=s(x+y,0)=s(x+y,x+x^{\bullet})=s(x,x)+s(y,x^{\bullet})\\
= s(y,x^{\bullet})+s(x,x) = s(y+x,x^{\bullet}+x)=s(y+x,0)=y+x.
\end{multline*}
Note that the addition of $M$ is uniquely determined by $x+y=s(x,s(0,y))$.

We finish by giving a simple alternative counterexample.

\begin{example}
We let $\WW$ be the free subtractive variety. Its objects---triples $(X,s,0)$ which satisfy $s(x,x)=0$ and $s(x,0)=x$ for all $x\in X$---are called \defn{subtraction algebras}. Consider the set $A=\{0,a,b\}$ equipped with the operation $s$ defined by the table
$$
\begin{array}{c|ccc}
 s & 0 & a & b \\ \hline
 0 & 0 & 0 & 0 \\
 a & a & 0 & 0 \\
 b & b & 0 & 0
\end{array}
$$
The internal relation $R=\{(0,0), (a,a), (b,b), (a,b)\}$ on the subtraction algebra~$A$ is reflexive and transitive, but not symmetric. Hence (B) is not an equivalence.
\end{example}

\subsection*{Acknowledgement}
We are grateful to the referee for his helpful comments and suggestions.


\providecommand{\noopsort}[1]{}
\providecommand{\bysame}{\leavevmode\hbox to3em{\hrulefill}\thinspace}
\providecommand{\MR}{\relax\ifhmode\unskip\space\fi MR }
\providecommand{\MRhref}[2]{%
  \href{http://www.ams.org/mathscinet-getitem?mr=#1}{#2}
}
\providecommand{\href}[2]{#2}

\end{document}